\documentclass[11pt,a4paper]{amsart}

\usepackage{amsmath}
\usepackage{amsfonts}
\usepackage{amssymb}

\usepackage{amsxtra}
\usepackage{amsthm}
\usepackage[dvipsnames]{xcolor}
\theoremstyle{plain}
\newcommand{\C}{\mathbb{C}}
\newcommand{\F}{\mathcal{F}}

\newtheorem{theorem}{Theorem}[section]
\newtheorem{maintheorem}{Theorem}

\newtheorem{secondtheorem}{Theorem}[section]

\newtheorem{lemma}[theorem]{Lemma}
\newtheorem{proposition}[theorem]{Proposition}
\newtheorem{corollary}[theorem]{Corollary}
\newtheorem{maincorollary}{Corollary}

\theoremstyle{definition}
\newtheorem{definition}{Definition}[section]
\newtheorem{example}{Example}[section]
\newtheorem{remark}{Remark}[section]

\DeclareSymbolFont{matha}{OML}{txmi}{m}{it}
\DeclareMathSymbol{\varv}{\mathord}{matha}{118}
\usepackage[hidelinks]{hyperref}
\usepackage{microtype}
\allowdisplaybreaks[1]
\date{}
\begin{document}

\title[The Bruce-Roberts Tjurina number of 1-forms along complex varieties]{The Bruce-Roberts Tjurina number of holomorphic 1-forms along complex analytic varieties}

\author{Pedro Barbosa}
\address[Pedro Barbosa]{Departamento de Matem\'atica - ICEX, Universidade Federal de Minas Gerais, UFMG}
\curraddr{Av. Pres. Ant\^onio Carlos 6627, 31270-901, Belo Horizonte-MG, Brasil.}
\email{pedrocbj@ufmg.br}
\author{Arturo Fern\'andez-P\'erez}
\address[A. Fern\'andez-P\'erez]{Departamento de Matem\'atica - ICEX, Universidade Federal de Minas Gerais, UFMG}
\curraddr{Av. Pres. Ant\^onio Carlos 6627, 31270-901, Belo Horizonte-MG, Brasil.}
\email{fernandez@ufmg.br}

\author{V\'ictor Le\'on}
\address[V. Le\'on]{ILACVN - CICN, Universidade Federal da Integra\c c\~ao Latino-Americana, UNILA}
\curraddr{Itaipu Parquetec, Foz do Igua\c cu-PR, 85867-970 - Brasil}
\email{victor.leon@unila.edu.br}

\thanks{The first author is supported by CAPES-Brazil. The second author acknowledges support from CNPq Projeto Universal 408687/2023-1 ``Geometria das Equa\c{c}\~oes Diferenciais Alg\'ebricas" and CNPq-Brazil PQ-306011/2023-9.}

\begin{abstract}
We introduce the notion of the \textit{Bruce-Roberts Tjurina number} for holomorphic 1-forms with respect to a pair $(X,V)$ of complex analytic subvarieties. We characterize its finiteness and, when $X$ and $V$ are hypersurfaces with at most isolated singularities, relate it to a determinantal colength, the Tjurina number of the 1-form with respect to $V$, and the Tjurina number of $X$. We also establish a decomposition of the difference between the Bruce-Roberts Milnor and Tjurina numbers and a bound for their quotient. As an application, we prove a relative quasihomogeneity result for germs of non-dicritical generalized curve foliations in complex dimension two.
\end{abstract}
\dedicatory{In memory of Celso dos Santos Viana}
\maketitle

\section{Introduction and statement of the results}
Let $(X,0)$ denote the germ of a reduced complex analytic variety in $(\C^n,0)$, $n\geq2$, and let $\omega$ be a germ of a holomorphic 1-form with an isolated singularity at the origin. We denote by $\mathcal{O}_n$ the local ring of holomorphic function germs at $0\in\C^n$, by $\mathfrak m_n$ its maximal ideal, and by $\Theta_n$ the $\mathcal{O}_n$-module of holomorphic vector field germs. Following K. Saito \cite{Saito1980}, the module of \textit{logarithmic vector fields} along $X$ is
\[
\Theta_X=\{\xi\in\Theta_n:\xi(I_X)\subseteq I_X\},
\]
where $I_X$ is the ideal of functions vanishing on $X$. Geometrically, these are the vector fields tangent to the regular part of $X$.

For a reduced complex analytic germ $(V,0)$, we define the \textit{Bruce-Roberts Tjurina number} of $\omega$ with respect to $(X,V)$ by
\[
\tau_{BR}(\omega,X,V):=\dim_{\C}
\frac{\mathcal{O}_n}{\omega(\Theta_X)+I_V},
\]
allowing the value $\infty$. This definition makes sense for arbitrary $V$ and does not require invariance. Let $\Sigma_X(\omega)$ denote the zero set of the coherent ideal $\omega(\Theta_X)$. The local analytic Nullstellensatz gives
\begin{equation}\label{finite-tau}
\tau_{BR}(\omega,X,V)<\infty
\quad\Longleftrightarrow\quad
V\cap\Sigma_X(\omega)=\{0\}\quad\text{as germs}.
\end{equation}
If $X$ has at most an isolated singularity, then away from the origin this condition says that
\[
\omega_p|_{T_pX}\neq0\qquad
\text{for every }p\in (X_{\mathrm{reg}}\cap V)\setminus\{0\}
\text{ sufficiently close to }0.
\]
Indeed, outside $X$ the logarithmic module is $\Theta_n$, and on $X_{\mathrm{reg}}$ its values span the tangent space of $X$. The isolated singularity assumption on $\omega$ excludes other zeros outside $X$.

In the geometric applications, $V=\{f=0\}$ will be a reduced hypersurface invariant by $\omega$. We recall that invariance means $T_pV\subseteq\ker\omega_p$ at every regular point of $V$. It is equivalent to divisibility of the coefficients of $\omega\wedge df$ by $f$. The hypersurface assumption allows us to use one defining equation $f$, the classical Tjurina number, and multiplication by $f$. Invariance is needed for the interpretation in terms of separatrices and foliation indices in Section~\ref{fol}; the algebraic identities below do not require it unless it is explicitly stated.

If $\omega=df$ and $V=\{f=0\}$, then $\tau_{BR}(df,X,V)$ coincides with the Bruce-Roberts Tjurina number $\tau_{BR}(f,X)$ studied in \cite{Ahmed2011,Bivia2020,Bivia2024,Ruas2024}. In particular,
\[
\tau_{BR}(df,\C^n,V)
=\dim_\C\frac{\mathcal{O}_n}{\langle f_{x_1},\ldots,f_{x_n},f\rangle}
=\tau_0(V),
\]
which is the classical Tjurina number \cite{Tjurina}.

The \textit{Bruce-Roberts number} of $\omega$ relative to $X$ was introduced in \cite{Barbosa2024} as
\[
\mu_{BR}(\omega,X):=\dim_\C\frac{\mathcal{O}_n}{\omega(\Theta_X)}.
\]
For an exact form, this is the Bruce-Roberts number $\mu_{BR}(f,X)$ of Bruce-Roberts \cite{BR}. 
Directly from the definitions,
\[
\tau_{BR}(\omega,X,V)\leq\mu_{BR}(\omega,X).
\]
When $X$ has at most an isolated singularity, finiteness of $\mu_{BR}(\omega,X)$ means that the restriction of $\omega$ to $X_{\mathrm{reg}}$ has no zeros in a punctured neighborhood of the origin. For a reduced plane curve, it is equivalent to requiring that none of its irreducible components be invariant by $\omega$. Non-invariance of $X$ as a whole is not sufficient when $X$ is reducible; in higher dimensions, non-invariance alone is also insufficient.

Finiteness in \eqref{finite-tau} can hold even when $\mu_{BR}$ is infinite. For example, for $\omega=y\,dx+x\,dy$, $X=\{x=0\}$ and $V=\{y=0\}$, one has $\omega(\Theta_X)=\langle x\rangle$, so that
\[
\tau_{BR}(\omega,X,V)=1,\qquad \mu_{BR}(\omega,X)=\infty.
\]
The stronger hypothesis $\mu_{BR}(\omega,X)<\infty$ will be imposed whenever we use the Tjurina number of $X$ through evaluation on logarithmic vector fields.

The relative Bruce--Roberts number $\mu^-_{BR}(\omega,X)$ for
holomorphic $1$-forms was introduced in our earlier work
\cite[p.~4]{Barbosa2024}. Lima-Pereira, Nu\~no-Ballesteros,
Or\'efice-Okamoto and Tomazella \cite[Section~2]{Lima2024a}
study this invariant for isolated complete intersection
singularities. Taking $V=X$ in the definition above gives
\[
\tau_{BR}(\omega,X,X)=\mu^-_{BR}(\omega,X).
\]

For $\omega=\sum_{j=1}^n A_j\,dx_j$, we write
\[
\langle\omega\rangle:=\langle A_1,\ldots,A_n\rangle,
\qquad
\mu_0(\omega):=\dim_\C\frac{\mathcal{O}_n}{\langle\omega\rangle},
\]
and define the \textit{Tjurina number of $\omega$ with respect to $V=\{f=0\}$} by
\begin{equation}\label{tjurina}
\tau_0(\omega,V):=\dim_\C
\frac{\mathcal{O}_n}{\langle A_1,\ldots,A_n,f\rangle}.
\end{equation}
For vector fields tangent to a hypersurface, related quotients occur in G\'omez-Mont \cite{Xavier1998}. In dimension two, the terminology Tjurina number for foliations is used in \cite{Cano}; see also \cite{Licanic2004,FPGBSM,FPGBSM24}.

Let $X=\phi^{-1}(0)$ be a positive-dimensional isolated complete intersection singularity (ICIS), where $\phi=(\phi_1,\ldots,\phi_k)$. Ebeling-Gusein-Zade \cite{Ebeling2001,Ebeling} introduced the $GSV$-index of a 1-form on $X$. When the restricted 1-form has an isolated zero, its algebraic expression is
\[
\operatorname{Ind}_{\operatorname{GSV}}(\omega;X,0)
=\dim_\C\frac{\mathcal{O}_n}{I_X+I_{k+1}\begin{pmatrix}\omega\\d\phi\end{pmatrix}},
\]
where $I_{k+1}$ denotes the ideal of maximal minors of
\begin{equation}\label{eq_1}
\begin{pmatrix}\omega\\d\phi\end{pmatrix}
=\begin{pmatrix}
A_1&\cdots&A_n\\
\partial\phi_1/\partial x_1&\cdots&\partial\phi_1/\partial x_n\\
\vdots&&\vdots\\
\partial\phi_k/\partial x_1&\cdots&\partial\phi_k/\partial x_n
\end{pmatrix}.
\end{equation}
See also \cite[Theorem 5.3.35]{Ebeling2023}. We use the notation
\begin{equation}\label{gsv}
\operatorname{Ind}_{\operatorname{GSV}}(\omega;X,V,0)
:=\dim_\C\frac{\mathcal{O}_n}
{I_X+I_{k+1}\begin{pmatrix}\omega\\d\phi\end{pmatrix}+\langle f\rangle}.
\end{equation}
This is a determinantal colength associated with the pair $(X,V)$. The notation extends that of the usual $GSV$-index; we do not assert a topological index interpretation for \eqref{gsv}. It is finite, in particular, when $\mu_{BR}(\omega,X)<\infty$.

For the ICIS $X=\phi^{-1}(0)$, we denote by
\[
\Theta_X^T:=\ker d\phi+I_X\Theta_n
\]
the submodule of \textit{trivial vector fields}. For a hypersurface $X=\{\phi=0\}$, its generators are given by \cite[Proposition 2.1]{Lima2024}:
\[
\phi\frac{\partial}{\partial x_i},\qquad
\phi_{x_i}\frac{\partial}{\partial x_j}
-\phi_{x_j}\frac{\partial}{\partial x_i},
\qquad 1\leq i\leq n,\quad 1\leq i<j\leq n.
\]
For $X=\{\phi=0\}$, we also set
\[
J:=I_2\begin{pmatrix}\omega\\d\phi\end{pmatrix},
\qquad
\kappa(\omega,X,V):=\dim_\C
\frac{\langle\omega\rangle+\big((J+\langle f\rangle):\phi\big)}
{\langle\omega\rangle+\langle f\rangle}.
\]
Here, $(K:\phi)=\{a\in\mathcal{O}_n:\phi a\in K\}$ is the ideal quotient of $K$ by $\langle\phi\rangle$. Since $\langle f\rangle\subseteq(J+\langle f\rangle):\phi$, the quotient is well defined and has finite dimension.

\begin{maintheorem}\label{2mainresult}
Let $\omega$ be a germ of a holomorphic 1-form with an isolated singularity at $0\in\C^n$, $n\geq2$. Let $X=\{\phi=0\}$ and $V=\{f=0\}$ be reduced hypersurface germs with at most isolated singularities, and assume that $\mu_{BR}(\omega,X)<\infty$. Then
\begin{align*}
\tau_{BR}(\omega,X,V)
={}&\operatorname{Ind}_{\operatorname{GSV}}(\omega;X,V,0)
+\tau_0(\omega,V)-\tau_0(X)\\
&+\dim_\C\frac{\omega(\Theta_X)\cap\langle f\rangle}
{\omega(\Theta_X^T)\cap\langle f\rangle}
-\kappa(\omega,X,V).
\end{align*}
\end{maintheorem}
The term $\kappa$ measures the failure of a multiplication map to remain injective after passing to the quotient by $f$. In particular, it vanishes precisely when
\begin{equation}\label{ideal-condition}
(J+\langle f\rangle):\phi\ \subseteq\ \langle\omega\rangle+\langle f\rangle.
\end{equation}

Lima-Pereira, Nu\~no-Ballesteros, Or\'efice-Okamoto and Tomazella \cite[Proposition 3.3]{Lima2024a} express the Bruce-Roberts number of a 1-form on an ICIS through its determinantal algebra and the evaluation quotient. Our proof uses the same basic exact-sequence framework. Passing to the quotient by $f$ introduces the intersection quotient and the term $\kappa$. These terms explain why the formula for $\mu_{BR}$ cannot be transferred to $\tau_{BR}$ by simply adding $\langle f\rangle$ to each denominator. Section~\ref{prova} gives a complete proof.

\begin{maincorollary}\label{corollary1}
Let $X\subset(\C^n,0)$ be a reduced hypersurface germ with at most an isolated singularity, and let $f\in\mathfrak m_n$ have an isolated critical point and satisfy $\mu_{BR}(f,X)<\infty$. Set $V=\{f=0\}$. Then
\begin{align*}
\tau_{BR}(f,X)
={}&\operatorname{Ind}_{\operatorname{GSV}}(df;X,V,0)
+\tau_0(V)-\tau_0(X)\\
&+\dim_\C\frac{df(\Theta_X)\cap\langle f\rangle}
{df(\Theta_X^T)\cap\langle f\rangle}
-\kappa(df,X,V).
\end{align*}
\end{maincorollary}
\begin{proof}
Apply Theorem~\ref{2mainresult} to $\omega=df$ and use $\tau_0(df,V)=\tau_0(V)$.
\end{proof}

Our application to foliations is motivated by \cite[Theorem 4.1]{Bivia2024} and Mattei's theorem \cite[Th\'eor\`eme A]{quasiMattei}.
\begin{secondtheorem}\label{propfol}
Let $\F$ be a germ of a non-dicritical generalized curve foliation at $(\C^2,0)$, defined by a 1-form $\omega$ with an isolated singularity. Let $X$ be a reduced curve germ such that none of its irreducible components is invariant by $\F$, and let $f$ be a reduced equation for the union $V$ of all separatrices of $\F$. If
\[
\mu_{BR}(\F,X)=\tau_{BR}(\F,X,V),
\]
then there exist coordinates $(u,v)$ centered at the origin, $g,h\in\mathcal{O}_2$ with $g(0)\neq0$, and positive integers $\alpha,\beta,d_0$ such that
\[
f(u,v)=\sum_{\alpha i+\beta j=d_0}P_{i,j}u^iv^j,
\qquad
 g\omega=df+h(\beta v\,du-\alpha u\,dv),
\]
and the pair $(f,X)$ is relatively quasihomogeneous in these coordinates.
\end{secondtheorem}
The foliation notation and its independence of the defining 1-form are explained in Section~\ref{fol}. Generalized curve foliations were introduced in \cite{CLS}.

The paper is organized as follows. In Section~\ref{basic}, we present the algebraic tools used in the proof of Theorem~\ref{2mainresult}. Section~\ref{prova} contains that proof and examples. In Section~\ref{bruce}, we relate the Bruce-Roberts Milnor and Tjurina numbers. Finally, Section~\ref{fol} gives the applications to foliations and the proof of Theorem~\ref{propfol}.

\section{Basic tools}\label{basic}
We first record the precise multiplication sequence needed below.
\begin{lemma}\label{newlema61}
Let $K,P\subseteq\mathcal{O}_n$ be ideals and $g\in\mathcal{O}_n$. There is an exact sequence
\[
0\longrightarrow\frac{\mathcal{O}_n}{P+(K:g)}
\overset{\cdot g}{\longrightarrow}
\frac{\mathcal{O}_n}{K+gP}
\longrightarrow\frac{\mathcal{O}_n}{K+\langle g\rangle}
\longrightarrow0.
\]
If $K+gP$ has finite colength, then
\[
\dim_\C\frac{\mathcal{O}_n}{K+gP}
=\dim_\C\frac{\mathcal{O}_n}{P+(K:g)}
+\dim_\C\frac{\mathcal{O}_n}{K+\langle g\rangle}.
\]
If, in addition, $g$ is a non-zero-divisor on $\mathcal{O}_n/K$, then $(K:g)=K$, and the first summand is $\dim_\C\mathcal{O}_n/(P+K)$.
\end{lemma}
\begin{proof}
The first map sends $a$ to $ga$. If $ga\in K+gP$, write $ga=k+gp$, with $k\in K$ and $p\in P$. Then $a-p\in(K:g)$, which proves injectivity and identifies its kernel before taking the quotient. The last map is the natural projection; its kernel is generated by the class of $g$, hence is the image of the first map. Taking dimensions proves the formula. The last assertion follows from the definition of a non-zero-divisor.
\end{proof}
\begin{remark}
For $K=\langle f_1,\ldots,f_m\rangle$ and $P=\langle p_1,\ldots,p_n\rangle$, finiteness and pairwise relative primality of $g$ and the $f_i$ do not suffice to replace $(K:g)$ by $K$. For example, take $K=P=\langle x^2,y^2\rangle$ and $g=x+y$ in $\mathcal{O}_2$. The left-hand length is $4$, whereas the proposed sum with $K$ in place of $(K:g)$ would be $4+2$. In our application $K=J+\langle f\rangle$, so the ideal $(K:g)$ must be retained.
\end{remark}

We also use the following elementary module identity; see \cite[Proposition 2.1]{aatiyah}.
\begin{lemma}\label{re1}
Let $R$ be a ring, $M$ an $R$-module, and $M_1,M_2,M_3$ submodules of $M$, with $M_3\subseteq M_2$. Then
\[
\frac{M_1+M_2}{M_1+M_3}
\cong\frac{M_2}{M_3+(M_1\cap M_2)}.
\]
\end{lemma}
\begin{proof}
Restrict the quotient map $M_1+M_2\to(M_1+M_2)/(M_1+M_3)$ to $M_2$. It is surjective and its kernel is $M_3+(M_1\cap M_2)$.
\end{proof}

We recall the following description of the module $\Theta_X^T$ defined in the introduction; see \cite[Proposition 2.1]{Lima2024}.
\begin{proposition}\label{prop12}
Let $X=\phi^{-1}(0)$ be an ICIS, with $\phi=(\phi_1,\ldots,\phi_k)$. Then $\Theta_X^T$ is generated by the $(k+1)$-minors of
\[
\begin{pmatrix}
\partial/\partial x_1&\cdots&\partial/\partial x_n\\
\partial\phi_1/\partial x_1&\cdots&\partial\phi_1/\partial x_n\\
\vdots&&\vdots\\
\partial\phi_k/\partial x_1&\cdots&\partial\phi_k/\partial x_n
\end{pmatrix}
\]
and the vector fields $\phi_i\partial/\partial x_j$, $1\leq i\leq k$, $1\leq j\leq n$. Consequently,
\begin{equation}\label{eq_omega}
\omega(\Theta_X^T)
=I_{k+1}\begin{pmatrix}\omega\\d\phi\end{pmatrix}
+\langle\phi_iA_j:1\leq i\leq k,\ 1\leq j\leq n\rangle.
\end{equation}
\end{proposition}

For an ICIS $X$ with $\mu_{BR}(\omega,X)<\infty$, evaluation by
$\omega$ induces the following isomorphism
\cite[Theorem~3.6]{Lima2024a}.
\begin{equation}\label{eq_tu}
\frac{\Theta_X}{\Theta_X^T}
\overset{\cdot\omega}{\cong}
\frac{\omega(\Theta_X)}{\omega(\Theta_X^T)},
\qquad
\dim_{\C}\frac{\Theta_X}{\Theta_X^T}=\tau_0(X).
\end{equation}
The dimension identity follows from
\cite[Proposition~4.6]{Lima2024} and is recorded in
\cite[Corollary~3.7]{Lima2024a}. For isolated hypersurface
singularities, the corresponding exact sequence is stated in
\cite[Lemma~4.2]{Barbosa2024}. When $X$ is smooth,
$\Theta_X=\Theta_X^T$ and $\tau_0(X)=0$.

\section{Proof of Theorem \ref{2mainresult}}\label{prova}
\begin{proof}[Proof of Theorem \ref{2mainresult}]
Set $A=\langle\omega\rangle$, $B=\omega(\Theta_X)$, $T=\omega(\Theta_X^T)$, and $F=\langle f\rangle$. The exact sequence
\[
0\longrightarrow\frac{B+F}{T+F}
\longrightarrow\frac{\mathcal{O}_n}{T+F}
\longrightarrow\frac{\mathcal{O}_n}{B+F}\longrightarrow0
\]
gives
\begin{equation}\label{eq_prin}
\tau_{BR}(\omega,X,V)
=\dim_\C\frac{\mathcal{O}_n}{T+F}
-\dim_\C\frac{B+F}{T+F}.
\end{equation}
All terms have finite dimension: by \eqref{eq_tu}, $B/T$ is finite-dimensional, while $\mathcal{O}_n/B$ has finite dimension by hypothesis.

By Lemma~\ref{re1} and the usual quotient sequence,
\[
\frac{B+F}{T+F}\cong\frac{B}{T+(B\cap F)},
\qquad
\frac{T+(B\cap F)}{T}\cong\frac{B\cap F}{T\cap F}.
\]
Consequently, \eqref{eq_tu} yields
\begin{equation}\label{eq_inter}
\dim_\C\frac{B+F}{T+F}
=\tau_0(X)-\dim_\C\frac{B\cap F}{T\cap F}.
\end{equation}

Now write $J=I_2\begin{pmatrix}\omega\\d\phi\end{pmatrix}$ and $K=J+F$. By \eqref{eq_omega}, $T=J+\phi A$. Lemma~\ref{newlema61}, applied to $g=\phi$ and $P=A$, gives
\begin{equation}\label{eq_ideal_exact}
0\longrightarrow\frac{\mathcal{O}_n}{A+(K:\phi)}
\overset{\cdot\phi}{\longrightarrow}
\frac{\mathcal{O}_n}{T+F}
\longrightarrow\frac{\mathcal{O}_n}{J+\langle\phi,f\rangle}
\longrightarrow0.
\end{equation}
Since $F\subseteq(K:\phi)$, the definition of $\kappa$ gives
\[
\dim_\C\frac{\mathcal{O}_n}{A+(K:\phi)}
=\tau_0(\omega,V)-\kappa(\omega,X,V).
\]
Using \eqref{gsv} in \eqref{eq_ideal_exact}, we obtain
\begin{equation}\label{eq_10}
\dim_\C\frac{\mathcal{O}_n}{T+F}
=\operatorname{Ind}_{\operatorname{GSV}}(\omega;X,V,0)
+\tau_0(\omega,V)-\kappa(\omega,X,V).
\end{equation}
Substitution of \eqref{eq_inter} and \eqref{eq_10} into \eqref{eq_prin} proves the theorem.
\end{proof}

\begin{remark}\label{remark-kappa}
One always has $0\leq\kappa(\omega,X,V)\leq\tau_0(\omega,V)$. The equality \eqref{eq_10} also shows that $\kappa$ is independent of multiplication of $\omega$, $\phi$, or $f$ by units: the three other lengths in that equality are unchanged. In particular, the formula without the term $\kappa$ is valid whenever \eqref{ideal-condition} holds. No non-zero-divisor assertion about $\phi$ modulo $J+\langle f\rangle$ is used in the proof.
\end{remark}

\begin{example}\label{ex-three}
Let $X=\{\phi=x^3+yz=0\}$ and $V=\{f=x^2+y^2+z^2=0\}$ in $(\C^3,0)$, and consider
\[
\omega=df+f(z\,dx+x\,dy+y\,dz).
\]
The linear part of its coefficient map is $2\operatorname{Id}$, so $\omega$ has an isolated singularity. A local standard basis computation gives $\mu_{BR}(\omega,X)=6$, so the finiteness hypothesis is satisfied. The hypersurface $V$ is invariant, since $\omega\wedge df=f(z\,dx+x\,dy+y\,dz)\wedge df$.

As $\phi$ is weighted homogeneous, $\Theta_X$ is generated by
\[
z\partial_x-3x^2\partial_y,\quad
y\partial_x-3x^2\partial_z,\quad
y\partial_y-z\partial_z,\quad
x\partial_x+2y\partial_y+z\partial_z.
\]
Modulo $f$, the minors are generated, up to nonzero constants, by
\[
xz-3x^2y,\qquad xy-3x^2z,\qquad y^2-z^2.
\]
With $E=x\partial_x+2y\partial_y+z\partial_z$, one has
\[
\omega(E)\equiv2x^2+4y^2+2z^2\pmod{f}.
\]
Thus $\omega(\Theta_X)+\langle f\rangle=\langle x^2,xy,xz,y^2,z^2\rangle$ after local ideal reduction, and
\[
\tau_{BR}(\omega,X,V)=5.
\]
The remaining lengths, obtained from these same ideals, are
\[
\operatorname{Ind}_{\operatorname{GSV}}(\omega;X,V,0)=5,
\quad \tau_0(\omega,V)=1,
\quad \tau_0(X)=2,
\quad\dim_\C\frac{\mathcal{O}_3}{\omega(\Theta_X^T)+\langle f\rangle}=6.
\]
Therefore $\kappa=0$, $\dim_\C(\omega(\Theta_X)+\langle f\rangle)/(\omega(\Theta_X^T)+\langle f\rangle)=1$, and \eqref{eq_inter} gives an intersection correction of dimension $1$. The formula reads $5=5+1-2+1$.
\end{example}

\begin{example}\label{ex1}
Let $p,q\geq2$ be integers, let $X=\{\phi=y^p-x^q=0\}$, $V=\{f=xy=0\}$, and set
\[
\omega=y\,dx+\lambda x\,dy,
\qquad \lambda\in\C\setminus\{0,-p/q\}.
\]
Then $\omega$ has an isolated singularity, $V$ is invariant, and no component of $X$ is invariant, since
\[
\omega\wedge d\phi=(py^p+\lambda qx^q)\,dx\wedge dy,
\qquad
\omega\wedge df=(1-\lambda)xy\,dx\wedge dy.
\]
The value $\lambda=1$ is allowed. The Euler and Hamiltonian generators of $\Theta_X$ are
\[
px\partial_x+qy\partial_y,
\qquad
py^{p-1}\partial_x+qx^{q-1}\partial_y;
\]
see \cite{Saito1980,Saito1}. Hence, with $D=py^p+\lambda qx^q$,
\[
\omega(\Theta_X)+\langle f\rangle=\langle xy,D\rangle
=\omega(\Theta_X^T)+\langle f\rangle.
\]
For the last equality, use $\lambda\neq0$ to see that $x^{q+1}$ and $y^{p+1}$ belong to $\langle xy,D\rangle$; these are precisely the additional generators coming from $\phi\langle\omega\rangle$ modulo $xy$.
A basis of $\mathcal{O}_2/\langle xy,D\rangle$ is
\[
1,x,\ldots,x^q,y,\ldots,y^{p-1},
\]
so that $\tau_{BR}(\omega,X,V)=p+q$. On the other hand,
\[
\langle\phi,D,xy\rangle=\langle x^q,y^p,xy\rangle,
\]
which has colength $p+q-1$. Since $\tau_0(\omega,V)=1$ and $\tau_0(X)=(p-1)(q-1)$, equations \eqref{eq_10} and \eqref{eq_inter} give
\[
\kappa=0,\qquad
\dim_\C\frac{\omega(\Theta_X)\cap\langle f\rangle}
{\omega(\Theta_X^T)\cap\langle f\rangle}=(p-1)(q-1).
\]
Thus Theorem~\ref{2mainresult} gives $p+q=(p+q-1)+1$.
\end{example}

\begin{example}
In $(\C^2,0)$, let $X=\{\phi=0\}$ and $V=\{f=0\}$, where
\[
\phi=y,\qquad f=x^2-y^2,\qquad
\omega=(x^2+y^2)\,dx-2xy\,dy.
\]
The coefficients of $\omega$ vanish simultaneously only at the origin.
The curve $V$ is a reduced node and is invariant by $\omega$, since
\[
\omega\wedge df=2yf\,dx\wedge dy.
\]
For the smooth curve $X=\{y=0\}$, we have
\[
\Theta_X=\Theta_X^T=\langle\partial_x,y\partial_y\rangle,
\qquad
\omega(\Theta_X)=\langle x^2+y^2,xy^2\rangle.
\]
The quotient by $\omega(\Theta_X)$ has basis
$1,x,y,xy,y^2,y^3$, so $\mu_{BR}(\omega,X)=6$.
Also, $\tau_0(X)=0$, and the equality $\Theta_X=\Theta_X^T$
makes the intersection quotient in Theorem~\ref{2mainresult} zero.

The determinantal ideal is $J=\langle x^2+y^2\rangle$. Adding
$\langle f\rangle$ gives
\begin{align*}
\omega(\Theta_X)+\langle f\rangle &= \langle x^2,y^2\rangle,\\
J+\langle\phi,f\rangle &= \langle x^2,y\rangle,\\
\langle\omega\rangle+\langle f\rangle &= \langle x^2,xy,y^2\rangle.
\end{align*}
Their colengths give
\[
\tau_{BR}(\omega,X,V)=4,\qquad
\operatorname{Ind}_{\operatorname{GSV}}(\omega;X,V,0)=2,\qquad
\tau_0(\omega,V)=3.
\]
Since $J+\langle f\rangle=\langle x^2,y^2\rangle$, we have
\[
(J+\langle f\rangle):\phi
=\langle x^2,y^2\rangle:y
=\langle x^2,y\rangle.
\]
It follows that
\[
\kappa(\omega,X,V)
=\dim_{\C}\frac{\langle x^2,y\rangle}{\langle x^2,xy,y^2\rangle}
=1,
\]
with the class of $y$ as a basis. Thus the equality in
Theorem~\ref{2mainresult} reads
\[
4=2+3-0+0-1.
\]
\end{example}
\section{The Bruce-Roberts Tjurina number and the Bruce-Roberts number}\label{bruce}
Let $(X,0)$ be a reduced complex analytic germ, and assume $\mu_{BR}(\omega,X)<\infty$. For a hypersurface germ $V=\{f=0\}$ with reduced defining
equation $f$, write
\[
\Theta_V^\omega:=\{\delta\in\Theta_n:\omega(\delta)\in\langle f\rangle\},
\qquad
H_\omega:=\{\zeta\in\Theta_n:\omega(\zeta)=0\}.
\]
In particular, $H_\omega\subseteq\Theta_V^\omega$. Following the framework of \cite{Bivia2024}, we define
\[
\overline\mu_X(\omega):=\dim_\C\frac{\Theta_n}{\Theta_X+H_\omega},
\qquad
\overline\tau_X(\omega,V):=\dim_\C
\frac{\Theta_n}{\Theta_X+\Theta_V^\omega}.
\]
The next proposition proves that both numbers are finite. Moreover,
\begin{equation}\label{bar-difference}
\overline\mu_X(\omega)-\overline\tau_X(\omega,V)
=\dim_\C\frac{\Theta_V^\omega}
{H_\omega+(\Theta_X\cap\Theta_V^\omega)}\geq0.
\end{equation}

\begin{proposition}\label{prop1}
Let $(X,0)$ be a reduced complex analytic germ and let $\omega$ satisfy $\mu_{BR}(\omega,X)<\infty$. For a hypersurface $V=\{f=0\}$, one has
\begin{align*}
\mu_{BR}(\omega,X)&=\mu_0(\omega)+\overline\mu_X(\omega),\\
\tau_{BR}(\omega,X,V)&=\tau_0(\omega,V)+\overline\tau_X(\omega,V).
\end{align*}
In particular,
\begin{align*}
\mu_{BR}(\omega,X)-\tau_{BR}(\omega,X,V)
={}&\mu_0(\omega)-\tau_0(\omega,V)\\
&+\overline\mu_X(\omega)-\overline\tau_X(\omega,V).
\end{align*}
\end{proposition}
\begin{proof}
Evaluation on $\omega$ gives the exact sequences
\[
0\longrightarrow\frac{\Theta_n}{\Theta_X+H_\omega}
\overset{\cdot\omega}{\longrightarrow}
\frac{\mathcal{O}_n}{\omega(\Theta_X)}
\longrightarrow\frac{\mathcal{O}_n}{\langle\omega\rangle}
\longrightarrow0
\]
and
\[
0\longrightarrow\frac{\Theta_n}{\Theta_X+\Theta_V^\omega}
\overset{\cdot\omega}{\longrightarrow}
\frac{\mathcal{O}_n}{\omega(\Theta_X)+\langle f\rangle}
\longrightarrow\frac{\mathcal{O}_n}{\langle\omega\rangle+\langle f\rangle}
\longrightarrow0.
\]
For example, the kernel condition in the second sequence says that $\omega(\xi)=\omega(\eta)+af$ for some $\eta\in\Theta_X$, which is equivalent to $\xi-\eta\in\Theta_V^\omega$. This proves injectivity. The last maps are the natural projections. Taking dimensions gives the two identities. Formula \eqref{bar-difference} follows by applying Lemma~\ref{re1} to the corresponding quotient modules.
\end{proof}

\begin{corollary}\label{coro4}
Under the hypotheses of Proposition~\ref{prop1}, the following conditions are equivalent:
\begin{enumerate}
\item $\mu_{BR}(\omega,X)=\tau_{BR}(\omega,X,V)$;
\item $f\in\omega(\Theta_X)$;
\item $\mu_0(\omega)=\tau_0(\omega,V)$ and $\overline\mu_X(\omega)=\overline\tau_X(\omega,V)$.
\end{enumerate}
The last equality in (3) is equivalent to
\[
\Theta_V^\omega=H_\omega+(\Theta_X\cap\Theta_V^\omega).
\]
\end{corollary}
\begin{proof}
The projection $\mathcal{O}_n/\omega(\Theta_X)\to\mathcal{O}_n/(\omega(\Theta_X)+\langle f\rangle)$ has kernel generated by the class of $f$. Its source has finite dimension, so equality of the dimensions is equivalent to the vanishing of that class. The equivalence with (3) follows from Proposition~\ref{prop1}, since both differences on its right-hand side are nonnegative. The last assertion follows from \eqref{bar-difference}.
\end{proof}
When $\omega=df$, this recovers \cite[Corollary 2.2]{Bivia2024}.

\begin{corollary}\label{bar-indices}
Let $\omega$ be a germ of a holomorphic $1$-form with an isolated
singularity at the origin in $(\C^n,0)$, $n\geq 2$. Let
$X=\{\phi=0\}$ and $V=\{f=0\}$ be reduced hypersurface germs, each
with at most an isolated singularity at the origin. Assume that
$\mu_{BR}(\omega,X)<\infty$. Then
\[
\overline{\mu}_X(\omega)
=\operatorname{Ind}_{\operatorname{GSV}}(\omega;X,0)-\tau_0(X)
\]
and
\[
\overline{\tau}_X(\omega,V)
=\operatorname{Ind}_{\operatorname{GSV}}(\omega;X,V,0)-\tau_0(X)
+\dim_{\C}
\frac{\omega(\Theta_X)\cap\langle f\rangle}
     {\omega(\Theta_X^T)\cap\langle f\rangle}
-\kappa(\omega,X,V).
\]
\end{corollary}
\begin{proof}
The first identity follows by comparing Proposition~\ref{prop1} with
\[
\mu_{BR}(\omega,X)
=\mu_0(\omega)+\operatorname{Ind}_{\operatorname{GSV}}(\omega;X,0)-\tau_0(X),
\]
proved in \cite[Theorem 1]{Barbosa2024}. It also follows from \cite[Proposition 3.3 and Corollary 3.7]{Lima2024a} in the hypersurface case. For the second identity, compare Proposition~\ref{prop1} with Theorem~\ref{2mainresult}.
\end{proof}

For $f\in\mathfrak m_n$, denote by $r_f(\omega(\Theta_X))$ the smallest integer $r\geq1$ such that $f^r\in\omega(\Theta_X)$. This integer is finite when $\mu_{BR}(\omega,X)<\infty$. The following bound is the length-theoretic argument underlying \cite[Theorem 3.2]{Bivia2020}.
\begin{corollary}\label{biviasruas}
Let $X$ be a reduced complex analytic germ, let $\mu_{BR}(\omega,X)<\infty$, and let $V=\{f=0\}$, with $f\in\mathfrak m_n$. Then
\[
\frac{\mu_{BR}(\omega,X)}{\tau_{BR}(\omega,X,V)}
\leq r_f(\omega(\Theta_X)).
\]
\end{corollary}
\begin{proof}
Put $Q=\mathcal{O}_n/\omega(\Theta_X)$ and $r=r_f(\omega(\Theta_X))$. Multiplication by $f$ on the finite-dimensional algebra $Q$ is nilpotent, and $f^rQ=0$. For each $j=0,\ldots,r-1$, multiplication by $f^j$ induces a surjection
\[
Q/fQ\longrightarrow f^jQ/f^{j+1}Q.
\]
The filtration $Q\supseteq fQ\supseteq\cdots\supseteq f^rQ=0$ therefore gives
\[
\mu_{BR}(\omega,X)
=\sum_{j=0}^{r-1}\dim_\C(f^jQ/f^{j+1}Q)
\leq r\dim_\C(Q/fQ)
=r\tau_{BR}(\omega,X,V).
\]
\end{proof}

\begin{example}\label{ex-ratio}
Let $m\geq1$ be an integer and consider
\[
f=x^{2m+1}+x^my^{m+1}+y^{2m},\qquad
V=\{f=0\},\qquad X=\{xy=0\},
\]
with
\[
\omega=(f_x+yf)\,dx+(f_y+xf)\,dy
=df+f\,d(xy).
\]
The curve $V$ is invariant. Since $\omega=e^{-xy}d(e^{xy}f)$ and $f$ is semiquasihomogeneous with principal part $x^{2m+1}+y^{2m}$, it follows that
\[
\mu_0(\omega)=\mu_0(f)=2m(2m-1).
\]
Neither axis is invariant by $\omega$. Modulo $xy$, the determinant associated with $d(xy)$ is $(2m+1)x^{2m+1}-2my^{2m}$, whose quotient together with $xy$ has colength $4m+1$. As $\tau_0(X)=1$, the formula used in Corollary~\ref{bar-indices} gives
\[
\mu_{BR}(\omega,X)=4m^2+2m.
\]
Moreover, $\Theta_X=\langle x\partial_x,y\partial_y\rangle$, and
\[
\omega(\Theta_X)+\langle f\rangle
=\langle xf_x,yf_y,f\rangle
=\langle x^{2m+1},x^my^{m+1},y^{2m}\rangle.
\]
The last equality follows by solving the three constant linear relations between these monomials; the coefficient matrix has nonzero determinant $m+1$. Counting monomials gives
\[
\tau_{BR}(\omega,X,V)
=(2m+1)2m-(m+1)(m-1)=3m^2+2m+1.
\]
Thus we obtain the following uniform expression:
\[
\frac{\mu_{BR}(\omega,X)}{\tau_{BR}(\omega,X,V)}
=\frac{4m^2+2m}{3m^2+2m+1}<\frac43.
\]
For $m=1$, both numbers are $6$ and Corollary~\ref{coro4} gives $r_f=1$. For $m\geq2$, their difference is $m^2-1>0$, so $r_f\geq2$ and
\begin{equation}\label{eq_curves}
\frac{\mu_{BR}(\omega,X)}{\tau_{BR}(\omega,X,V)}
<\frac43<2\leq r_f(\omega(\Theta_X)).
\end{equation}
This verifies Corollary~\ref{biviasruas} for every $m$, without requiring a separate computation of the nilpotency exponent.
\end{example}

\begin{remark}
For $\omega=df$ and $X=\C^2$, the quotient becomes $\mu_0(V)/\tau_0(V)$. The inequality $\mu_0(V)/\tau_0(V)<4/3$, for a reduced singular plane curve, was proposed by Dimca-Greuel \cite{Dimca-Greuel} and proved in full generality by Almir\'on \cite{Almiron}. Example~\ref{ex-ratio} is a relative family with the same limiting quotient; it does not establish a universal $4/3$ bound for 1-forms or foliations. For related inequalities for foliations, see \cite{FPGBSM24}.
\end{remark}

\section{Applications to holomorphic foliations in complex dimension two}\label{fol}
For background on holomorphic foliations, we refer to \cite{Brunella-book}. Let $\F:\omega=0$ be a germ of a singular foliation on $(\C^2,0)$, where $\omega=A\,dx+B\,dy$ has relatively prime coefficients and an isolated singularity. Any other reduced 1-form defining the same foliation is $u\omega$, where $u\in\mathcal{O}_2$ is a unit. Since $\omega(\Theta_X)$ is an ideal,
\[
(u\omega)(\Theta_X)=u\omega(\Theta_X)=\omega(\Theta_X).
\]
Consequently, the notation
\[
\mu_{BR}(\F,X):=\mu_{BR}(\omega,X),\qquad
\tau_{BR}(\F,X,V):=\tau_{BR}(\omega,X,V)
\]
is independent of the chosen reduced 1-form. The same argument applies to $\tau_0(\F,V)$, the determinantal colength \eqref{gsv}, the intersection quotient, and $\kappa$.

For a reduced invariant curve $V=\{f=0\}$, Saito's decomposition \cite{Saito1980} gives
\[
g\omega=h\,df+f\eta,
\]
where $g,h$ are relatively prime to $f$ and $\eta$ is a holomorphic 1-form. The foliation $GSV$-index is
\[
\operatorname{GSV}_0(\F,V)
=\frac{1}{2\pi i}\int_{\partial V}\frac gh\,d\left(\frac hg\right),
\]
where $\partial V$ is the oriented link. This index was introduced in \cite{GSV}; we follow the presentation in \cite{Brunella1997}. It satisfies
\begin{equation}\label{xav}
\operatorname{GSV}_0(\F,V)=\tau_0(\F,V)-\tau_0(V);
\end{equation}
see \cite{Xavier1998,FPGBSM}. The index in \eqref{xav} is attached to a tangent vector field defining $\F$ on the invariant curve $V$; it should be distinguished from the 1-form determinantal colength associated with $(X,V)$.

\begin{corollary}\label{fol-comparison}
Let $\F:\omega=0$ be a germ of a holomorphic foliation on
$(\C^2,0)$ with an isolated singularity. Let $f$ be a reduced
equation of an $\F$-invariant curve germ $V=\{f=0\}$ that is
singular at the origin. Let $X$ be a reduced curve germ such that
none of its irreducible components is invariant by $\F$. Then
\begin{align*}
&\tau_{BR}(\F,X,V)-\tau_{BR}(f,X)\\
&\quad=\operatorname{Ind}_{\operatorname{GSV}}(\F;X,V,0)
-\operatorname{Ind}_{\operatorname{GSV}}(df;X,V,0)
+\operatorname{GSV}_0(\F,V)\\
&\qquad+\dim_\C\frac{\omega(\Theta_X)\cap\langle f\rangle}
{\omega(\Theta_X^T)\cap\langle f\rangle}
-\dim_\C\frac{df(\Theta_X)\cap\langle f\rangle}
{df(\Theta_X^T)\cap\langle f\rangle}\\
&\qquad-\kappa(\omega,X,V)+\kappa(df,X,V).
\end{align*}
\end{corollary}
\begin{proof}
The hypothesis on $X$ ensures $\mu_{BR}(\omega,X)<\infty$. Also $X$ and $V$ have no common component, so $f$ is nonconstant on each branch of $X$ and $\mu_{BR}(f,X)<\infty$. Subtract Corollary~\ref{corollary1} from Theorem~\ref{2mainresult} and use \eqref{xav}.
\end{proof}

We now recall relative quasihomogeneity, following \cite{Bivia2024}.
\begin{definition}
A pair $(f,X)$ in $(\C^n,0)$ is \textit{relatively quasihomogeneous} if there are coordinates $x=(x_1,\ldots,x_n)$, positive rational weights $w=(w_1,\ldots,w_n)$, and generators $h_1,\ldots,h_m$ of $I_X$ such that $f$ has weighted degree $1$ and each $h_i$ has a positive weighted degree $d_i$. In these coordinates,
\[
f(x)=\sum_{\langle w,\nu\rangle=1}a_\nu x^\nu,
\qquad
h_i(x)=\sum_{\langle w,\nu\rangle=d_i}b_{\nu,i}x^\nu.
\]
\end{definition}
We use the criterion of \cite[Corollary 3.4]{Bivia2024}: it is enough to exhibit $\delta\in\Theta_X$ with $\delta(0)=0$, $\delta(f)=f$, and positive rational eigenvalues of its linear part. We verify all three conditions in the proof below. In particular, tangency and the eigenfunction condition alone would not suffice.

\begin{lemma}\label{theta-invariant}
Let $V=\{f=0\}\subset(\C^2,0)$ be a reduced invariant curve of a 1-form $\omega$ with an isolated singularity. Then $\Theta_V^\omega=\Theta_V$.
\end{lemma}
\begin{proof}
At every sufficiently close nonzero point of $V_{\mathrm{reg}}$, the forms $df$ and $\omega$ are nonzero and have the same one-dimensional kernel. Thus $\omega(\xi)$ vanishes on $V$ if and only if $df(\xi)$ does. Since $f$ is reduced, vanishing on the regular part of every branch is equivalent to divisibility by $f$. This proves the assertion.
\end{proof}

\begin{proof}[Proof of Theorem~\ref{propfol}]
Let $\omega=A\,dx+B\,dy$. Since no branch of $X$ is invariant by
$\F$, we have $\mu_{BR}(\omega,X)<\infty$. By
Corollary~\ref{coro4}, the equality in the statement gives
$f\in\omega(\Theta_X)\subseteq\langle A,B\rangle$ and
\begin{equation}\label{eq_8}
\Theta_V^\omega=H_\omega+(\Theta_X\cap\Theta_V^\omega).
\end{equation}
Since $\F$ is a non-dicritical generalized curve, Mattei's theorem
\cite[Th\'eor\`eme A]{quasiMattei} gives coordinates $(u,v)$,
positive integers $\alpha,\beta,d_0$, and $g,h\in\mathcal O_2$ with
$g(0)\neq0$ such that
\begin{align}
f(u,v)&=\sum_{\alpha i+\beta j=d_0}P_{i,j}u^iv^j,\label{eq_7}\\
g\omega&=df+h(\beta v\,du-\alpha u\,dv).\label{eq_9}
\end{align}
Let $X_{\alpha,\beta}=\alpha u\partial_u+\beta v\partial_v$.
By \eqref{eq_7} and \eqref{eq_9}, we have
$X_{\alpha,\beta}(f)=d_0 f$ and $\omega(X_{\alpha,\beta})=d_0 f/g$.
In particular, $X_{\alpha,\beta}\in\Theta_V^\omega$. By \eqref{eq_8},
there exists $\zeta\in H_\omega$ such that
\[
\delta=X_{\alpha,\beta}-\zeta\in\Theta_X\cap\Theta_V^\omega,
\qquad \frac{\omega(\delta)}{f}=\frac{d_0}{g}.
\]
We normalize $\delta$ by a unit to obtain $\delta(f)=f$ and determine
the eigenvalues of its linear part.

By \cite{CLS}, $\mu_0(\omega)=\mu_0(f)$, so $f$ is singular.
Suppose first that $\mu_0(f)>1$. The quadratic part of $f$ is then
zero or has rank one. In the latter case the weights are distinct,
since equal weights would make $f$ itself a quadratic square,
contrary to reducedness. The quadratic monomials have distinct
weighted degrees when $\alpha\ne\beta$, so a rank-one quadratic
part is a multiple of $u^2$ or $v^2$. In each of these cases the
determinant of the linear coefficient map of the right-hand side
of \eqref{eq_9} is $\alpha\beta h(0)^2$. Since $\mu_0(\omega)>1$,
this determinant vanishes and hence $h(0)=0$.

The coefficients of $g\omega$ are relatively prime, so for some
$a\in\mathcal O_2$ we have
\[
\zeta=a\bigl(f_v\partial_u-f_u\partial_v-hX_{\alpha,\beta}\bigr),
\qquad \delta(f)=d_0(1+ah)f.
\]
Divide $\delta$ by the unit $d_0(1+ah)$ and keep the notation $\delta$.
Its linear part is diagonal if the quadratic part of $f$ vanishes.
In the rank-one case it is triangular with distinct diagonal entries.
In both cases the linear part is semisimple with eigenvalues
$\alpha/d_0$ and $\beta/d_0$.

If $\mu_0(f)=1$, then $V$ is an ordinary node. Choose coordinates
with $f=uv$. Invariance of the axes gives
$\omega=v\,a\,du+u\,b\,dv$ for $a,b\in\mathcal O_2$.
By Lemma~\ref{theta-invariant}, $\delta\in\Theta_V$, so we can write
$\delta=uc\partial_u+ve\partial_v$. Then
$\omega(\delta)/f=ac+be$ is a unit, and $(c(0),e(0))\ne(0,0)$.
Choose a primitive parametrization $\gamma(t)=(u(t),v(t))$
of a branch of $X$. Since this branch is distinct from the axes,
both $r=\operatorname{ord}_t u(t)$ and $s=\operatorname{ord}_t v(t)$
are positive and finite. Tangency gives
\[
u(t)c(\gamma(t))v'(t)=v(t)e(\gamma(t))u'(t).
\]
Comparing coefficients of $t^{r+s-1}$ yields $s\,c(0)=r\,e(0)$.
Thus $c+e$ is a unit. Dividing $\delta$ by $c+e$ gives $\delta(f)=f$
and a diagonal linear part with eigenvalues $r/(r+s)$ and $s/(r+s)$.

In both cases the normalized field $\delta$ vanishes at the origin,
belongs to $\Theta_X$, satisfies $\delta(f)=f$, and has a semisimple
linear part with positive rational eigenvalues. Moreover,
$\omega(\delta)/f$ is still a unit. Since the eigenvalues of the
linear part of $\delta$ are positive, its spectrum lies in the
Poincar\'e domain. By the Poincar\'e normalization theorem
\cite[Theorem~5.5]{IlyashenkoYakovenko2008}, we can choose analytic
Poincar\'e--Dulac coordinates for $\delta$. We denote these coordinates
again by $(u,v)$ and write the eigenvalues as $\alpha/d_0$ and
$\beta/d_0$, with $\alpha,\beta,d_0$ positive integers. In these
coordinates the semisimple part of $\delta$ is
$\delta_S=X_{\alpha,\beta}/d_0$.
By \cite[Theorem 3.1(b)]{Bivia2024}, we have $\delta_S(f)=f$, so
\[
f(u,v)=\sum_{\alpha i+\beta j=d_0}P_{i,j}u^iv^j.
\]

Since $\delta$ and $\delta_S$ belong to $\Theta_V$,
Lemma~\ref{theta-invariant} gives
$\omega(\delta-\delta_S)\in\langle f\rangle$.
The linear part of $\delta-\delta_S$ vanishes. The algebraic
multiplicity of $\omega$ is $\operatorname{ord}_0 f-1$, since
$\F$ is a non-dicritical generalized curve
\cite[Proposition 2.4]{FPGBSM}. Hence $\omega(\delta-\delta_S)$
vanishes to order at least $\operatorname{ord}_0 f+1$.
It follows that $\omega(\delta-\delta_S)/f$ vanishes at the origin,
so $\omega(\delta_S)/f$ is a unit. Taking
$g=(\omega(\delta_S)/f)^{-1}$ gives
$(g\omega-df)(X_{\alpha,\beta})=0$. Since $u$ and $v$ are relatively
prime, we obtain
\[
g\omega=df+h(\beta v\,du-\alpha u\,dv)
\]
for some $h\in\mathcal O_2$, with $g(0)\ne0$.

Finally, we apply \cite[Corollary 3.4]{Bivia2024} to the normalized
field $\delta\in\Theta_X$. The proof of that corollary gives
quasihomogeneous generators $h_1,\ldots,h_m$ of $I_X$ of positive
weighted degrees for the weights $(\alpha/d_0,\beta/d_0)$ in the
Poincar\'e--Dulac coordinates $(u,v)$ chosen above. Since $f$ has
weighted degree $1$ for the same weights, the pair $(f,X)$ is
relatively quasihomogeneous in these coordinates.
\end{proof}
\begin{remark}
Take $\lambda=2$ in Example~\ref{ex1}, so that
$\omega=y\,dx+2x\,dy$. The associated vector field
$2x\partial_x-y\partial_y$ has a reduced nondegenerate singularity.
Every separatrix through the origin lies in the zero fiber of the
first integral $xy^2$, and both coordinate axes are invariant.
Thus $\F$ is a non-dicritical generalized curve whose union of
separatrices is $V=\{f=xy=0\}$.

The Euler field $X_{p,q}=px\partial_x+qy\partial_y$ belongs to
$\Theta_X$ and satisfies
\[
\omega(X_{p,q})=(p+2q)f.
\]
Hence $f\in\omega(\Theta_X)$, and Corollary~\ref{coro4} together
with the computation in Example~\ref{ex1} gives
\[
\mu_{BR}(\F,X)=\tau_{BR}(\F,X,V)=p+q.
\]

For the weights $(p,q)$, the functions $f=xy$ and $\phi=y^p-x^q$
have weighted degrees $p+q$ and $pq$, respectively. Thus the
normalized weights $(p/(p+q),q/(p+q))$ exhibit the relative
quasihomogeneity of $(f,X)$ in the coordinates $(x,y)$.

In Theorem~\ref{propfol}, we may therefore take $(u,v)=(x,y)$
and $(\alpha,\beta,d_0)=(p,q,p+q)$. The normalized Euler field is
$\delta_S=X_{p,q}/d_0$, which gives
$g=(\omega(\delta_S)/f)^{-1}=(p+q)/(p+2q)$. In these coordinates,
\[
g\omega=d(xy)-\frac{1}{p+2q}(qy\,dx-px\,dy),
\]
so the normal form holds with $h=-1/(p+2q)$.
\end{remark}

\end{document}